\documentclass[reqno]{amsart}
\usepackage{amssymb,amsmath,hyperref}
\usepackage{amsrefs, float}
\usepackage[foot]{amsaddr}
\usepackage{bbold,stackrel, multicol}
 
\newcommand{\Z}{{\textsf{\textup{Z}}}}

\newtheorem{thm}{Theorem}

\newtheorem{cor}[thm]{Corollary}
\newtheorem{defi}[thm]{Definition}

\newtheorem{rem}[thm]{Remark}
\newtheorem{nota}[thm]{Notation}

\newtheorem{princ}[thm]{Principle}

\newtheorem{ack}[thm]{Acknowledgement}

\newtheorem*{tempo*}{Template}

\newcommand\be{\begin{equation}}
\newcommand\ee{\end{equation}} 

\usepackage{amsmath,amsfonts} 

\usepackage[applemac]{inputenc}

\def\bdefi{\begin{defi}\rm}
\def\edefi{\end{defi}}
\def\bnota{\begin{nota}\rm}
\def\enota{\end{nota}}
\def\FIVE{\Pi_{1}^{1}\text{-\textup{\textsf{CA}}}_{0}}
\def\SIX{\Pi_{2}^{1}\text{-\textsf{\textup{CA}}}_{0}}
\def\SIXk{\Pi_{k}^{1}\text{-\textsf{\textup{CA}}}_{0}}
\def\SIXko{\Pi_{k+1}^{1}\text{-\textsf{\textup{CA}}}_{0}}
\def\SIXK{\Pi_{k}^{1}\text{-\textsf{\textup{CA}}}_{0}^{\omega}}

\def\ZFC{\textup{\textsf{ZFC}}}

 \def\r{\mathbb{r}}

\def\RCA{\textup{\textsf{RCA}}}
\def\({\textup{(}}
\def\){\textup{)}}

\def\RCAo{\textup{\textsf{RCA}}_{0}^{\omega}}
\def\ACAo{\textup{\textsf{ACA}}_{0}^{\omega}}

\def\WKL{\textup{\textsf{WKL}}}

\def\bye{\end{document}}
\def\N{{\mathbb  N}}
\def\Q{{\mathbb  Q}}
\def\R{{\mathbb  R}}

\def\SS{\textup{\textsf{S}}}

\def\di{\rightarrow}

\def\asa{\leftrightarrow}

\def\ACA{\textup{\textsf{ACA}}}

\def\BOOT{\textup{\textsf{BOOT}}}

\def\eps{\varepsilon}

\def\FF{\textup{\textsf{FF}}}

\def\INT{\textup{\textsf{int}}}

\usepackage{mathrsfs}  

\usepackage{graphicx}
\usepackage{tikz}
\usetikzlibrary{matrix, shapes.misc}
\usepackage{comment,tikz-cd}

\numberwithin{equation}{section}
\numberwithin{thm}{section}

\begin{document}
\title{Sometimes tame, sometimes wild: weak continuity}
\author{Sam Sanders}
\address{Department of Philosophy II, RUB Bochum, Germany}
\email{sasander@me.com}
\begin{abstract}
Continuity is one of the most central notions in mathematics, physics, and computer science.  
An interesting associated topic is \emph{decompositions of continuity}, where continuity is shown to be equivalent to the combination of two or more \emph{weak continuity} notions.  
In this paper, we study the logical properties of basic theorems about weakly continuous functions, like the supremum principle for the unit interval.  
We establish that most weak continuity notions are as tame as continuity, i.e.\ the supremum principle can be proved from the relatively weak \emph{arithmetical comprehension axiom} only.  
By contrast, for seven `wild' weak continuity notions, the associated supremum principle yields rather strong axioms, including Feferman's projection principle, full second-order arithmetic, or Kleene's associated quantifier $(\exists^{3})$.  Working in Kohlenbach's higher-order \emph{Reverse Mathematics}, we also obtain elegant equivalences in various cases and obtain similar results for e.g.\ Riemann integration.  
We believe these results to be of interest to mainstream mathematics as they cast new light on the distinction of `ordinary mathematics' versus `foundations of mathematics/set theory'.
\end{abstract}

\maketitle              
%


\section{Introduction}\label{gintro}
\subsection{Aim}\label{klintro}
It is a commonplace that continuity is one of the most central properties in mathematics, crucial as it is to analysis, topology, (theoretical) computer science, and physics.  
Now, continuity is \emph{very} well-studied in mathematical logic where it is classified as fairly \emph{tame}, as basic properties like the supremum principle on the unit interval, can be settled in rather \emph{weak} logical systems.  Now, there are dozens (perhaps hundreds?) of \textbf{decompositions of continuity}, where continuity is shown to be equivalent to the combination of two or more {weak continuity\footnote{We note that \emph{weak and generalised continuity} come with its own AMS code, namely 54C08, i.e.\ weak continuity is not a fringe topic in mathematics.} notions}, going back to Baire (\cite{beren2}, 1899), as follows:  
\begin{center}
\emph{continuity $\asa$ \big[weak continuity notion \textsf{\textup{A}} \textbf{\textup{and}} weak continuity notion \textsf{\textup{B}}}\big].  \textsf{(Deco)}
\end{center}
It is then a natural question whether these weak continuity notions are as \emph{tame} as continuity, say working over the real numbers.  
In this paper, we establish that most weak continuity notions are indeed as \emph{tame} as continuity, while seven `exceptional' ones are completely \emph{wild} in comparison.
Nonetheless, these tame and wild weak continuity notions exist side-by-side in the mainstream literature.  

\smallskip

To make the above italicised heuristic terms (tame, wild, hard) precise, we work in \emph{Reverse Mathematics}, a program where the aim is to find 
the minimal axioms needed to prove a given theorem of ordinary mathematics.  In particular, we show that for many weak continuity notions, basic properties like the supremum principle can be 
established in the kind of logical systems one uses to study continuity itself, namely \emph{arithmetical comprehension}.  What is surprising is that for other weak continuity notions, existing side-by-side in the literature, the associated basic properties, liking finding the supremum or evaluating the Riemann integral, are `wild' in that they  exist at the outer edges of Reverse Mathematics, implying as they do full second-order arithmetic in various incarnations.  In this paper, we study the supremum principle based on the following definitions of the supremum of a certain class $\Gamma$ of weakly continuous functions $f:[0,1]\di \R$ .
\begin{itemize}
\item[\textsf{(sup1)}] The supremum $\sup_{x\in [p, q]}f(x)$ exists as a sequence with variables over $p, q\in \Q\cap [0,1]$ for any given $f$ in $\Gamma$.
\item[\textsf{(sup2)}] The supremum $\lambda x.\sup_{y\in [0, 1]}f(x, y)$ exists as an $\R \di \R$-function with variable $x$ over the real numbers, for any given $f$ in $\Gamma$.  
\item[\textsf{(sup3)}] The supremum $\lambda f.\sup_{x\in [0, 1]}f(x)$ exists as an $(\R\di \R)\di \R$-functional with variable $f:[0,1]\di \R$ over the class $\Gamma$ of weakly continuous functions.
\end{itemize}
Each of these can be found in the literature (\cite{taoana2, taocompa, rudin, rudinrc, stein1}).  In particular, the supremum norm, e.g.\ on the Lebesgue measurable functions, denoted 
\be\label{normy}
\|f\|_{\infty}:= \sup_{x\in [0,1]}f(x),
\ee
is essentially $\lambda f.\sup_{x\in [0,1]}f(x)$.  
Moreover, \emph{maximal functions} from harmonic analysis have the form $\sup_{y\in M}f(x,y)$ (see e.g.\ \cite{stein1}*{p.\ 92} or \cite{stein4}*{p.\ 198 and 208}) and occur in the Hardy-Littlewood theorem (\cite{stein2}*{p.\ 48} and \cite{stein3}*{p.~246}).  On a historical note, Darboux (\cite{darb}*{p.\ 61, Theorem 1}, 1875) already establishes the existence of the supremum for \emph{arbitrary} bounded functions on bounded intervals.  We shall therefore use \emph{Darboux's supremum principle} to refer to the formulation using \textsf{(sup1)}.

\smallskip

We expand on the above in Section \ref{intro}, including a detailed overview of this paper.  Basic definitions can be found in Section~\ref{seca} while our main results are in Section~\ref{main}.
In particular, the so-called tame (resp.\ wild) weak continuity notions are analysed in Section \ref{klef2} (resp.\ Section \ref{klef} and \ref{lef}).

\smallskip

Finally, we believe our results to be of general interest to the mathematical community as they cast new light upon the distinction `ordinary mathematics' versus `set theory'.  
In particular, the received view based on the latter juxtaposition seems to be in serious need of adjustment, as discussed in detail in Section \ref{funda}.  

\subsection{Motivation and overview}\label{intro}
We provide an overview of and motivation for this paper, including \emph{Reverse Mathematics}, abbreviated RM 
and where Stillwell's excellent \cite{stillebron} furnishes an introduction to RM for the mathematician-in-the-street. 

\smallskip

First of all, RM is a program in the foundations of mathematics, founded by Friedman (\cites{fried, fried2}) and developed extensively by Simpson (\cites{simpson1, simpson2}), where the aim is to find the minimal axioms that prove a given theorem of ordinary mathematics.  
Once the minimal axioms are known, the theorem at hand is generally always \emph{equivalent} to the axioms over a weak logical system called the \emph{base theory}, as already observed by Friedman in the early days of RM as follows.
\begin{quote}
When the theorem is proved from the right axioms, the axioms can be proved from the theorem (\cite{fried}*{p.\ 235}).
\end{quote}
We assume basic familiarity with RM, including the \emph{higher-order} approach in \cite{kohlenbach2}, which is \emph{intimately} related to the Friedman-Simpson framework, as shown in \cite{dagsamXIV}.
In his pioneering paper \cite{kohlenbach2}, Kohlenbach studies `uniform' theorems and their equivalences; for instance, the following are equivalent over the base theory by \cite{kohlenbach2}*{\S3}.
\begin{itemize}
\item Kleene's $(\exists^{2}):$ $(\exists E:\N^{\N}\di \N)(\forall f\in \N^{\N})(E(f)=0\asa (\exists n\in \N)(f(n)=0)$.
\item \emph{There exists a discontinuous $\R\di \R$-function.}
\item \textsf{UWKL:} \emph{there exists a functional $\Phi:\N^{\N}\di \N^{\N}$ such that for $T$ an infinite binary tree, $\Phi(T)$ is a path through $T$.}
\item \emph{there exists $\Psi:(\R\di \R)\di \R$ such that for continuous $f:[0,1]\di \R$, $f$ attains its maximum at $\Psi(f)$, i.e.\ $(\forall x\in [0,1])(f(x)\leq f(\Psi(f)))$.}
\end{itemize}
The informed reader will recognise the final two items as the `uniform' versions of weak K\"onig's lemma and Weierstrass' maximum theorem, well-known from RM (see e.g.\ \cite{simpson2}*{IV.1-2}).
Kleene's quantifier $(\exists^{2})$ is the higher-order version of $\ACA_{0}$ and similar equivalences may be found in \cite{yamayamaharehare, jeffcarl, dagsamXIV}, including for the \emph{Suslin functional}, i.e.\ the higher-order version of $\FIVE$.  Kohlenbach also studies the following uniform version of the supremum principle in \cite{kohlenbach2}*{\S3}, which is our starting point:
\begin{center}
 \emph{there exists $\Xi:(\R\di \R)\di \R$ such that for continuous $f:[0,1]\di \R$, $\Xi(f)$ is the supremum of $f$ on $[0,1]$.}
\end{center}
Clearly, the previous supremum principle makes use of \textsf{(sup3)} from Section \ref{klintro} for continuous functions.  It is then a natural question what happens if we generalise `continuity' to `weak continuity', with the latter stemming from the literature.

\smallskip

Secondly, the above axioms, including $(\exists^{2})$, are rather weak compared to the `upper limit' of Friedman-Simpson RM, i.e.\ second-order arithmetic $\Z_{2}$.  To reach the latter, we establish equivalences between the following in Section~\ref{klef}:  
\begin{enumerate}
\renewcommand{\theenumi}{\alph{enumi}}
\item Kleene's $(\exists^{3}):$ $(\exists E)(\forall Y: \N^{\N}\di \N)(E(Y)=0\asa (\exists f\in \N^{\N})(Y(f)=0)$,\label{fenrier}
\item  \emph{there exists $\Xi:(\R\di \R)\di \R$ such that for $f:[0,1]\di \R$ in $\Gamma$, $\Xi(f)$ is the supremum of $f$ on $[0,1]$.}\label{gop}
\end{enumerate}
where $(\exists^{3})$ readily implies $\Z_{2}$ and where $\Gamma$ consists of certain \emph{weak continuity} notions from the literature, some going back more than a hundred years, namely having been studied in Young (\cite{jonnypalermo}) and Blumberg (\cite{bloemeken}).  We shall use the adjective `wild' to refer to weak continuity notions for which item (b) is equivalent to $(\exists^{3})$.

\smallskip

\emph{By contrast}, for other related weak continuity notions, the existence of the supremum functional as in item (b) can be proved using \emph{nothing more than} $(\exists^{2})$ and weaker axioms (Section \ref{klef2}).  We shall use the adjective `tame' to refer weak continuity notions for which item (b) follows from $(\exists^{2})$.  
As it happens, all (wild and tame) weak continuity notions are all taken from \emph{decompositions of continuity}, which take the form \textsf{(Deco)} from Section \ref{klintro}
for certain spaces and where the weak continuity notions are generally independent; we shall provide ample references to the relevant instances of \textsf{(Deco)} in the literature.
Moreover, the supremum principle is not special: we also establish equivalences involving $(\exists^{3})$ and other functionals, e.g.\ based on Riemann integration, the intermediate value property, and continuity (almost) everywhere.  We expect many more examples to exist.

\smallskip

Next, the above equivalences involve \emph{fourth-order} objects, like the supremum functional $\Xi$ in item \eqref{gop} and the comprehension functional $E$ from $(\exists^{3})$ in item~\eqref{fenrier}.  
In Section~\ref{lef1}, we study the \emph{third-order} supremum principle based on~\textsf{(sup1)} from Section~\ref{klintro}.  In particular, we obtain equivalences between Darboux's supremum principle for certain wild continuity notions and the axiom $\BOOT$.  The latter is essentially Feferman's \emph{projection principle} \textsf{Proj1} from \cite{littlefef}, a highly impredicative and explosive\footnote{The axiom $\BOOT$ is \emph{explosive} in that combining it with $\SIXK$ from Section \ref{lef1}, one obtains $\Pi_{k+1}^{1}$-comprehension by Theorem \ref{timo}.} principle.   

Finally, \textsf{(sup1)} and \textsf{(sup3)} from Section \ref{klintro} deal with the unit interval $[0,1]$.  We study the supremum principle based on \textsf{(sup2)} in Section \ref{beyo}.  In particular, we show that that the third-order supremum principle for wild weakly continuous $[0,1]^{2}\di \R$-functions, implies $\Pi_{n}^{1}\text{-\textsf{\textup{CA}}}_{0}$ for any $n>0$.
We establish the same for generalisations of \textsf{Proj1} to $n$-quantifier alternations.

\smallskip

In conclusion, we believe our equivalences for $(\exists^{3})$ to be surprising and unique: on one hand, while Hunter establishes equivalences for $(\exists^{3})$ in topology (\cite{hunterphd}), 
our supremum functional for (historical) wild weak continuity notions seems \emph{much} more basic.  
On the other hand, tame weak continuity notions have a supremum functional that can be handled by $(\exists^{2})$; the difference between the latter and $(\exists^{3})$ amounts to the difference between arithmetical comprehension $\ACA_{0}$ and full second-order comprehension $\Z_{2}$, which is \emph{unheard of} in RM, to the best of our knowledge.  Moreover, the supremum principle based on \textsf{(sup2)} is \emph{third-order} and yields $\Pi_{n}^{1}\text{-\textsf{\textup{CA}}}_{0}$ for any $n>0$, while the operator $\lambda x.\sup_{y\in [0,1]}f(x, y)$ may be found in graduate textbooks and throughout the literature.  We discuss the foundational implications of our results in some detail in Section \ref{funda}. 

\subsection{Basic definitions}\label{seca}
We introduce some definitions needed in the below, mostly stemming from mainstream mathematics.
We note that subsets of $\R$ are given by their characteristic functions (Definition \ref{char}), where the latter are common in measure and probability theory.

\smallskip
\noindent
First of all, we make use the usual definition of (open) set, where $B(x, r)$ is the open ball with radius $r>0$ centred at $x\in \R$.
\bdefi[Set]\label{char}~
\begin{itemize}
\item Subsets $A$ of $ \R$ are given by their characteristic function $F_{A}:\R\di \{0,1\}$, i.e.\ we write $x\in A$ for $ F_{A}(x)=1$ for all $x\in \R$.
\item We write `$A\subset B$' if we have $F_{A}(x)\leq F_{B}(x)$ for all $x\in \R$. 
\item A set $O\subset \R$ is \emph{open} if $x\in O$ implies that there is $k\in \N$ with $B(x, \frac{1}{2^{k}})\subset O$.
\item A set $C\subset \R$ is \emph{closed} if the complement $\R\setminus C$ is open. 
\item For $S\subset \R$, $\INT(S)$, $\partial S$, and $\overline{S}$ are the \emph{interior}, \emph{boundary}, and \emph{closure} of $S$. 
\end{itemize}
\edefi
\noindent
No computational data/additional representation is assumed for open/closed sets.  
As established in \cites{dagsamXII, dagsamXIII, dagsamXIV, samBIG, samBIG2}, one readily comes across closed sets in basic real analysis (Fourier series) that come with no additional representation. 

\smallskip

Secondly, we list an important remark about a certain notational practice of RM, also connected to Definition \ref{char}.
\begin{rem}[Virtual existence]\label{hio}\rm
It is common in RM to make use of a kind of `virtual' or `relative' existence (see e.g.\ \cite{simpson2}*{X.1}).  
As an example, $\ACA_{0}$ is equivalent to the existence of the supremum $\sup C$ of any (coded) closed $C\subset [0,1]$ by \cite{simpson2}*{IV.2.11}.   
Thus, in the weaker system $\RCA_{0}$, $\sup C$ may not exist for all closed $C\subset [0,1]$.  Nonetheless, the formula `$\sup C<1$' \emph{is} meaningful 
in $\RCA_{0}$, namely as shorthand for $(\exists k\in \N)(\forall x\in C)(x\leq 1-\frac{1}{2^{k}})$. 

\smallskip

Similarly, we note that while the interior, boundary, or closure of a set may not exist in weak systems like $\RCAo$, the formulas `$x\in \INT (S)$', `$x\in \partial S$', and `$x\in \overline{S}$'
make perfect sense as shorthand for the associated standard definitions in the language of all finite types.   In fact, the `closure' notation $\overline{S}$ is already in use in second-order RM for the notion of \emph{separably closed set}, which is the closure of a countable set of reals $S$, as can be gleaned from \cite{browner}*{p.\ 51}.
\end{rem}

\subsection{Foundational musings}\label{funda}
We discuss the foundational implications of our results, especially as they pertain to the distinction of `ordinary mathematics' versus `set theory'.  
In a nutshell, the results in this paper show that the received view based on the latter juxtaposition is in dire need of adjustment.  

\smallskip

First of all, the official foundations of (almost all of) mathematics are provided by \emph{Zermelo-Fraenkel set theory with the Axiom of Choice}, abbreviated $\ZFC$. 
Now, Hilbert was a strong proponent of set theory, but he also realised that large parts of mathematics can be developed in logical systems \emph{far} weaker than $\ZFC$.   
Together with Bernays, Hilbert undertook such a development in the \emph{Grundlagen} volumes (\cites{hillebilly, hillebilly2}).
For instance, a considerable swath of real analysis is developed in the main logical system $H$ from \cite{hillebilly2}*{Supplement IV}.
The program \emph{Reverse Mathematics} from Section~\ref{klintro} is viewed by its founders and many of its practitioners as a continuation of the work by Hilbert and Bernays (see \cite{simpson2, sigohi} for this opinion).  
 
\smallskip

Secondly, the focus of Reverse Mathematics is (therefore) on \emph{ordinary mathematics}, where the latter is qualified by Simpson as follows in \cite{simpson2}*{I.1}. 
\begin{quote}
On the one hand, there is set-theoretic mathematics, and on the other hand there is what we call ``non-set-theoretic'' or ``ordinary' mathematics.
By set-theoretic mathematics we mean those branches of mathematics that were created by
the set-theoretic revolution which took place approximately a century ago.
We have in mind such branches as general topology, abstract functional
analysis, the study of uncountable discrete algebraic structures, and of
course abstract set theory itself.

\smallskip

We identify as \emph{ordinary} or \emph{non-set-theoretic} that body of mathematics which is prior to or independent of the introduction of abstract set-theoretic concepts. We have in mind such branches as geometry, number theory, calculus, differential equations, real and complex analysis, count-
able algebra, the topology of complete separable metric spaces, mathematical logic, and computability theory.
\end{quote}
Thirdly, we repeat that the goal of Reverse Mathematics is to identify the minimal axioms needed to prove a theorem of ordinary mathematics, the latter having been described in the previous quote. 
Following many case studies, one observes that most of ordinary mathematics (in the above sense of Simpson) can be developed in rather weak logical systems, even compared to the Hilbert-Bernays system $H$ (\cite{simpson2, sigohi}).  In particular, many (some would say `most') theorems of ordinary mathematics can be proved in the logical system $\FIVE$ (see \cite{simpson2}*{p.\ 33} and \cite{montahue} for this opinion).  Feferman has made similar observations regarding `scientifically applicable' mathematics (\cite{fefermandas, fefermanlight}).  The following message now comes to the fore: 
\begin{center}
\emph{ordinary mathematics generally does not go much beyond\footnote{In fact, the known examples of (second-order) theorems that reach the stronger system $\SIX$, deal with topology (\cites{mummyphd, mummymf, mummy}).} the system $\FIVE$. } 
\end{center}
Let us refer to the previous centred statement as the \emph{logical} characterisation of ordinary mathematics.    

\smallskip

Fourth, following Corollary \ref{taitdied}, the supremum principle for \emph{tame} weakly continuous functions is provable in a weak system, one that is even finitistically reducible in the sense of Hilbert (see \cite{sigohi} and \cite{simpson2}*{X.3.18}).  Thus, these supremum principles would be part of `ordinary mathematics'.   By contrast, the results in Sections \ref{klef} and \ref{lef} establish that the supremum principles for \emph{wild} weakly continuous functions imply strong axioms, including Feferman's projection principle, full second-order arithmetic, and $(\exists^{3})$.  The latter go so far beyond $\FIVE$ that they would not be considered `ordinary mathematics', following the aforementioned logical characterisation of the latter.   

\smallskip

In conclusion, if we accept the above logical characterisation of ordinary mathematics, then we must also accept the following fundamental divide: 
\begin{center}
\emph{among the weak continuity notions that exist side-by-side in the literature, some give rise to ordinary mathematics, while others give rise to set theory.}
\end{center}
In our opinion, the previous centred statement is not acceptable and we therefore call into question the aforementioned logical characterisation of ordinary mathematics.   
At the very least, one has to admit the logical characterisation does not seem suitable for third- or fourth-order statements.  

\section{Main results}\label{main}
\noindent 
In this section, we establish the results sketched in Section \ref{intro} as follows, where we recall \textsf{(sup1)}-\textsf{(sup3)} from Section \ref{klintro} and $\BOOT$ from Section \ref{intro}.
The latter is our version of Feferman's projection principle \textsf{Proj1} from \cite{littlefef}.
\begin{itemize}
\item In Section \ref{klef}, we establish equivalences between Kleene's quantifier $(\exists^{3})$ and the supremum functional as in \textsf{(sup3)} for the `wild' weak continuity notions as in Definition \ref{KY}, including \emph{graph continuity}.
\item We show in Section \ref{klef2} that Kleene's quantifier $(\exists^{2})$ implies the existence of the supremum functional as in \textsf{(sup3)} for the many `tame' weak continuity notions from Definition \ref{WC}.
\item We establish equivalences between $\BOOT$ and Darboux's supremum principle based on \textsf{(sup1)} for wild weakly continuity notions (Section~\ref{lef1}).  We obtain similar results for related notions, including Riemann integrability.
\item The third-order supremum principle as in \textsf{(sup2)} for wild weakly continuous $[0,1]^{2}\di \R$-functions, implies $\Pi_{n}^{1}\text{-\textsf{\textup{CA}}}_{0}$ and generalisations of \textsf{Proj1} to $n$-quantifier alternations (Section \ref{beyo}).
\end{itemize}
%
%
In the below, we assume basic familiarity with Kohlenbach's {higher-order RM}.  The latter was introduced in \cite{kohlenbach2} and shown to be intimately related to Friedman-Simpson RM in \cite{dagsamXIV}. 
In particular, Kohlenbach's higher-order base theory $\RCAo$ from \cite{kohlenbach2} proves the same second-order sentences as the second-order base theory $\RCA_{0}$ of Friedman-Simpson RM (see \cite{kohlenbach2}*{\S2}). 

\smallskip

Finally, many mathematical notions have multiple equivalent definitions.  It goes without saying that 
these equivalences cannot always be proved in the base theory (and much stronger systems), where the most basic example is perhaps
continuity (see \cite{kohlenbach4}*{\S4}).  Similarly, many of the well-known inclusions from the hierarchy of function classes cannot always be proved in the base theory (and much stronger systems), as explored in some detail in \cite{dagsamXIV}*{\S2.8} and \cites{samBIG, samBIG2}.  

\subsection{Equivalent to Kleene's quantifier $(\exists^{3})$}\label{klef}
We establish the equivalences involving $(\exists^{3})$ sketched in Section \ref{intro}, 
including the supremum functional for weak continuity notions as in Def.~\ref{KY}.  
To show that the latter notions can also be \emph{tame}, we establish some of their fundamental properties in $\RCAo$, like the fact that \textbf{any} function on the reals is the sum of two graph continuous functions (see Theorem~\ref{lonk3}).  
We recall that $(\exists^{3})$ is the higher-order version of second-order arithmetic $\Z_{2}$, the `upper limit' of Friedman-Simpson RM. 

\smallskip

First of all, we shall study the weak continuity notions as in Def.\ \ref{KY}.   As an example of a decomposition of continuity as in \textsf{(Deco)}, continuity is equivalent to the combination of the Young condition (or: peripheral continuity \cite{gibs}) and having a closed graph (\cite{dasnest}), and to the combination of quasi-continuity (see Def.~\ref{WC}) and graph continuity (\cite{sack1}).  
The remaining decompositions may be found in e.g.\ \cites{tong2, smithje, gibs, haiti}. 
We recall Remark \ref{hio} regarding virtual existence and shall assume $(\exists^{2})$ for Definition~\ref{KY}, i.e.\ the latter takes place in $\ACAo\equiv \RCAo+(\exists^{2})$.    
\bdefi[Weak continuity I]\label{KY} For $f:[0,1]\di \R$, we have that
\begin{itemize}
\item the \emph{graph} of $f$ is defined as $G(f)=\{(x, y)\in \R\times \R: f(x)=_{\R}y \}$,
\item $f$ is \emph{graph continuous} if there is continuous $g:\R\di \R$ with $G(g)\subset \overline{G(f)}$,
\item $f$ is \emph{almost continuous} (Husain) if for any $x\in [0,1]$ and open $G\subset \R$ containing $f(x)$, the set $\overline{f^{-1}(G)}$ is a neighbourhood of $x$,
\item $f$ is \emph{almost continuous} (Stallings) if for any open $O\subset \R^{2}$ such that $G(f)\subset O$, there is a continuous $g:[0,1]\di \R$ such that $G(f)\subset O$. 
\item $f$ is said to be of \emph{Ces\`aro type} if there are non-empty open $U\subset [0,1], V\subset \R$ such that for all $y\in V$, $U\subset \overline{f^{-1}(y)}$. 
\item $f$ has the \emph{Young condition} in case for $x\in [0,1]$ there are sequences $(x_{n})_{n\in \N}, (y_{n})_{n\in \N}$ on the left and right of $x$ with the latter as limit and $\lim_{n\di \infty}f(x_{n})=f(x)=\lim_{n\di \infty}f(y_{n})$, 
\item $f$ is \emph{peripherally continuous} if for any $x\in [0,1]$ and open intervals $U, V$ with $x\in U, f(x)\in V$, there is an open $W\subset U$ with $x\in W$ and $f(\partial(W)) \subset V$.
\item $f$ is \emph{pre-continuous} if for any open $G\subset \R$, $f^{-1}(G)$ is pre-open\footnote{A set $S\subset \R$ is \emph{pre-open} in case $S\subset \INT (\overline{S})$.}.
\item $f$ is \emph{$\mathcal{C}$-continuous} if for any open $G\subset \R$, $f^{-1}(G)$ is a $C$-set\footnote{A set $S\subset \R$ is a $\mathcal{C}$-set in case $S=O\cap A$ where $O$ is open and $A$ satisfies $A\subset \INT (\overline{A})$}.
\end{itemize}
\edefi
The Heaviside function is quasi-continuous but not graph continuous, and similarly basic counterexamples exist for the other classes.
On a historical note, the Young condition can be found in \cite{jonnypalermo} while almost continuity (Husain) is studied by Blumberg (\cite{bloemeken}) for Euclidian space.  
Moreover, Husain's notion yields a characterisation of certain topological spaces (\cites{redherring, blof}) while Stalling's notion hails from fixed point theory (\cite{valtstil}). 
Continuity is also equivalent to the \emph{triple} combination of the two almost continuity notions (Stallings and Husain) and the `not of Ces\`aro type' property (\cite{smithje}) from Def.\ \ref{KY}.

\smallskip

Secondly, we shall study the supremum functional for $\Gamma$ equal to certain weak continuity notions from Def.\ \ref{KY}.
In case the function at hand does not have an upper (resp.\ lower) bound on a given interval, we use the special value $+\infty$ (resp.\ $-\infty$) for the supremum (infimum).  The set $\overline{\R}:=\R\cup \{+\infty, -\infty\}$ is called the \emph{set of extended reals} and we assume\footnote{To be absolutely clear, the formula `$x=y$' is $\Pi_{1}^{0}$ for $x, y\in \R$ by definition (see \cites{simpson2, kohlenbach2}) and we assume the same complexity for $x, y\in \overline{\R}$.\label{triffo}} $-\infty <x<+\infty$ by fiat.
\bdefi\label{KZ}
Any $\Xi:(\R\di \R)\di \overline{\R}$ is a \emph{supremum functional for $\Gamma$} if for any $f:[0,1]\di \R$ in $\Gamma$, $\Xi(f)$ is the supremum of $f$ on $[0,1]$. 
\edefi
It is part of the folklore of Kleene computability theory that the supremum functional for \emph{all} bounded functions on $[0,1]$ yields $(\exists^{3})$.   
It is interesting that the same holds for well-known function classes as in Def.\ \ref{KY}. 
\begin{thm}[$\RCAo$]\label{lonk} The following are equivalent.
\begin{itemize}
\item Kleene's $(\exists^{3})$: $(\exists E)(\forall Y: \N^{\N}\di \N)(E(Y)=0\asa (\exists f\in \N^{\N})(Y(f)=0)$,
\item Kleene's quantifier $(\exists^{2})$ plus the existence of a supremum functional for $\Gamma$ equal to the either: the Young condition, almost continuity \(Husain\), 
graph continuity, not of Ces\`aro type, peripheral, pre-, or $\mathcal{C}$-continuity. 
\end{itemize}
\end{thm}
\begin{proof}
First of all, we note (the well-known fact) that $\exists^{2}$ can convert a real $x\in [0,1]$ to a unique\footnote{In case of two binary representations, chose the one with a tail of zeros.} binary representation.  
Similarly, an element of Baire space $f\in \N^{\N}$ can be identified with its graph as a subset of $\N\times \N$, which can be represented as an element of Cantor space.  
In particular, $(\exists^{3})$ is equivalent to e.g.\ the following:
\be\label{alti}
(\exists E:(\R\di \R)\di \N)(\forall f: \R\di \R)(E(f)=0\asa (\exists x\in [0,1])(f(x)=_{\R}0)),
\ee
say over $\ACAo$ and perhaps weaker systems.  For the downward implication in the theorem, \eqref{alti} provides a decision procedure for $(\exists x\in [0,1])(f(x)>q)$, i.e.\ the usual interval-halving technique yields the supremum functional required for the second item of the theorem.  

\smallskip

Thirdly, for the upward implication, let $f:\R\di \R$ be given and use $\exists^{2}$ to define $g_{f}:\R\di \R$ as follows
\be\label{kruk}
g_{f}(x):=
\begin{cases}
1 & (\exists q\in \Q)( x+q \in [0,1]\wedge f(x+q)=0)\\
0 & \textup{ otherwise}
\end{cases}.
\ee
Note that we may assume that $f(q)\ne 0$ for $q\in \Q$ as $\exists^{2}$ can decide this case. 
Clearly, we have $(\exists x\in [0,1])(f(x)=0)\asa 1= \sup_{y\in [0,1]} g_{f}(y)$, i.e.\ the supremum functional for $g_{f}$ provides the required instance of \eqref{alti} for $f$. 
What remains to show is that $g_{f}$ satisfies the conditions from the theorem, as follows. 
\begin{itemize} 
 \item For the Young condition, fix $x_{0}\in [0,1]$ and note that in case $g_{f}(x_{0})=0$, take sequences of rationals converging to $x_{0}$.  In case $g_{f}(x_{0})\ne0$ take $(x_{0}\pm\frac{1}{2^{n}})_{n\in \N}$. 
 In each case, the Young condition is satisfied by basic sequences. 
\item For graph continuity, note that $g_{f}(q)=0$ for $q\in\Q$, i.e.\ $\R\times \{0\}\subset \overline{G(g_{f})} $.  
Hence, the `zero-everywhere' function $g_{0}$ is such that $G(g_{0})\subset \overline{G(g_{f})}$, i.e.\ $g_{f}$ is graph continuous and the associated continuous function is given.  
\item For almost continuity (Husain), note that for open $G\subset \R$ with $1\in G$ or $0\in G$, we have $\overline{g_{f}^{-1}(G)}=[0,1]$, which is a neighbourhood of any point.   
\item For the condition \emph{not of Ces\`aro type}, note that the latter notion expresses that for all non-empty open $U\subset [0,1], V\subset \R$ there is $y\in V$ with $U\not\subset \overline{g_{f}^{-1}(y)}$. 
Since $V$ is open, we can always choose $y \ne 0,1$, yielding $g_{f}^{-1}(y)=\emptyset$, implying that $g_{f}$ is not of Ces\'aro type.     
\item For pre-continuity, in case $g_{f}$ is zero everywhere, we are done.  In case $g(f)$ is not zero everywhere, fix open $G\subset \R$.  In case $1\in G$ or $0\in G$, $\overline{g_{f}^{-1}(G)}=[0,1]$, which shows that $g_{f}^{-1}(G)$ is pre-open in this case.  
In case $0,1 \not \in G$, $g_{f}^{-1}(G)=\emptyset$, which is also (pre-)open, and $g_{f}$ is pre-continuous
\item For peripheral continuity, in case $g_{f}(x_{0})=0$, take $W$ equal to a small enough interval with rational end-points containing $x_{0}$.  In case $g_{f}(x_{0})=1$, take $W$ equal to be $B(x_{0}, \frac{1}{2^{k}})$ for $k$ small enough.  
\item For $\mathcal{C}$-continuity, consider the following case distinction for open $G\subset \R$.
\begin{itemize}
\item If $0, 1\in G$, then $g_{f}^{-1}(G)=[0,1]$, which is a $\mathcal{C}$-set.
\item If $0\not \in G, 1\in G$, then $\overline{( g_{f}^{-1}(G))}=[0,1]$ and $g_{f}^{-1}(G)$ is a $\mathcal{C}$-set. 
\item If $0, 1\not \in G$, then $ g_{f}^{-1}(G)=\emptyset$, implying that $g_{f}^{-1}(G)$ is a $\mathcal{C}$-set. 
\item If $0\in G, 1\not \in G$, then $\overline{( g_{f}^{-1}(G))}=[0,1]$ and $g_{f}^{-1}(G)$ is a $\mathcal{C}$-set. 
\end{itemize}
\end{itemize}
The proof is now finished in light of the previous case distinction.
\end{proof}
We could also obtain the theorem for $D(c, \alpha)$-continuity (\cite{prezemski1}) with (mostly) the same case distinction as for $\mathcal{C}$-continuity in the previous proof.  
Presumably, there are other weak continuity notions in the literature we have not studied. 

\smallskip

Thirdly, we study certain decision functionals for weak continuity as in Def.~\ref{KY}.  
\bdefi\label{kwonk}
Any $\Xi:(\R\di \R)\di \R$ is a \emph{$P$-decision functional for $\Gamma$} if for any $f:[0,1]\di \R$ in $\Gamma$, $\Xi(f)=0$ if and only if $f$ satisfies the property $P$. 
\edefi
Similar to the supremum functional, a $P$-decision functional for \emph{all} (bounded) functions on $[0,1]$ readily yields $(\exists^{3})$ for most non-trivial properties $P$.   
It is interesting that the same holds for well-known function classes as in Def.\ \ref{KY} and well-known properties from real analysis. 
\begin{thm}[$\RCAo$]\label{lonk2}~
\begin{itemize}
\item Kleene's quantifier $(\exists^{3})$,
\item Kleene's $(\exists^{2})$, plus the existence of a $P$-decision functional for
\begin{itemize}
\item  $\Gamma$ equal to: the Young condition, almost continuity \(Husain\), graph continuity, not of Ces\`aro type, or pre-continuity. 
\item $P$ equal to: almost continuity \(Stallings\), continuity \(almost\) everywhere, Riemann integrability, or the intermediate value property.
\end{itemize}
\end{itemize}
\end{thm}
\begin{proof}
First of all, to show that $(\exists^{3})$ implies the other principles, we note that $(\exists^{3})$ can convert open sets to RM-codes (\cite{simpson2}*{II.5.6}) by \cite{dagsamVII}*{Theorem 4.5}. 
Hence, quantifying over open sets amounts to quantifying over real numbers.  Similarly, by \cite{dagsamXIV}*{Theorem 2.3}, quantifying over continuous functions on $\R$ similarly amounts to quantifying over real numbers, as the former can be represented by RM-codes (\cite{simpson2}*{II.6.1}) given weak K\"onig's lemma.  It is now a straightforward but tedious verification that $(\exists^{3})$ can decide the properties $P$ in the theorem.  

\smallskip

Secondly, to derive $(\exists^{3})$ from the other principles, consider $g_{f}$ from \eqref{kruk} and note that $(\exists x\in [0,1])(f(x)=0)$ if and only if $g_{f}$ is discontinuous (almost) everywhere.  
Similarly, $(\exists x\in [0,1])(f(x)=0)$ if and only $g_{f}$ is not Riemann integrable on $[0,1]$; one readily shows that $g_{f}$ is not Riemann integrable if it is not zero everywhere, in the same 
way as for Dirichlet's function $\mathbb{1}_{\Q}$.  The same argument works for the intermediate value property while the proof for almost continuity (Stallings) is fairly straightforward:  
 in case $f(x_{0})=0$ for $x_{0}\in [0,1]$, consider the open set $O\subset [0,1]\times \R$ consisting of $[0,1]\times \R$ with the line $y=1/2$ removed, with the line $x = 1/2$ removed for $y \leq 1/2$, and the line $x = x_{0}$ removed for $y \geq 1/2$.

Clearly, $G(g_{f})\subset O$ but there is no continuous function $g:[0,1]\di \R$ such that $G(g)\subset O$.  In this way, $(\forall x\in [0,1])(f(x)\ne 0)$ if and only if $g_{f}$ is almost continuous (Stallings), as required.
\end{proof}
In light of the aforementioned `triple decomposition' of continuity, it is pleasing to involve all three weak continuity notions in the previous theorem.  

\smallskip

Finally, we establish some fundamental properties of graph continuity in the base theory.  Our goal is dual, namely to show that graph continuity can be rather tame 
and that the same can hold for properties of \emph{all} functions on the reals.  To be absolutely clear, the following result is known (see \cite{grande1}) and the surprise lies in the fact that it can be proved in $\RCAo$.  Indeed, similar theorems with more narrow scope -like the Jordan decomposition theorem (\cite{dagsamXI})- cannot be proved in $\Z_{2}^{\omega}$.
\begin{thm}[$\RCAo$]\label{lonk3} For any $f:\R\di \R$, we have that
\begin{itemize}
\item  there are graph continuous $g, h:\R\di \R$ such that $f=g+h$, 
\item $f$ is continuous on $\R$ if and only if $f$ is quasi- and graph continuous on $\R$.
\end{itemize}
\end{thm}
\begin{proof}
For the first item, in case the function $f:\R\di \R$ is continuous, we can trivially take $g(x)=h(x)=f(x)/2$.  
In case $f:\R\di \R$ is \emph{dis}continuous, we obtain $(\exists^{2})$ by \cite{kohlenbach2}*{Prop.\ 3.12}. 
Using $(\exists^{2})$, the proof from \cite{grande1}*{Theorem 1} goes through as follows: let $A, B$ be countable dense sets in $\R$ with $A\cap B=\emptyset$, e.g.\ $A=\Q$ and $B=\{ q+\pi : q\in \Q\}$ works.  
Now define $f_{0}, f_{1}:\R\di \R$ as follows:
\be\label{insight}
f_{0}(x)
:=
\begin{cases}
0 & x\in A\\
f(x) & x\in B\\
f(x)/2 &\textup{otherwise}
\end{cases}
\textup{ and }
f_{1}(x)
:=
\begin{cases}
f(x) & x\in A\\
0 & x\in B\\
f(x)/2 &\textup{otherwise}
\end{cases}
\ee
and note that $f=f_{0}+f_{1}$.  Since $A$ and $B$ are dense, the `zero everywhere function' $g_{0}$ is such that $G(g_{0})\subset\overline{G(f_{i}) }$ for $i\leq 1$, i.e.\ $f_{0}, f_{1}$ are graph continuous.  

\smallskip

For the final item, the forward direction is trivial.  For the reverse direction, let $f:\R\di \R$ be quasi-continuous and graph continuous.  
By definition, there is continuous $g:\R\di \R$ such that $G(g)\subset \overline{G(f)}$.   We show that $g(x)=_{\R}f(x)$ for all $x\in \R$, establishing the continuity of $f$.  
Towards a contradiction, suppose $x_{0}\in \R$ is such that $|g(x_{0})-f(x_{0})|>\frac{1}{2^{k_{0}}}$ for $k_{0}\in \N$.
Now, by the continuity of $g$, there is $M_{0}\in \N$ such that for $y\in B(x_{0}, \frac{1}{2^{M_{0}}})$, we have $|g(x_{0})-g(y))|<\frac{1}{2^{k_{0}+3}}$. 
By the quasi-continuity of $f$, there is $N_{0}\geq M_{0}$ such that there is $(a, b)\subset B(x_{0},\frac{1}{2^{N_{0}}} )$ with $(\forall z\in (a, b))( |f(x_{0})-f(z)|<\frac{1}{2^{k_{0}+3}} )$.
Since $G(g)\subset \overline{G(f)}$, there is $w\in (a, b)$ such that $|g(\frac{a+b}{2})-f(w)|<\frac{1}{2^{k_{0}+3}}$.  
However, this implies 
\[\textstyle
|f(x_{0})-g(x_{0})| \leq |f(x_{0})-f(w)| +|f(w)-g(\frac{a+b}{2})|+    |g(\frac{a+b}{2})-g(x_{0})|   < \frac{3}{2^{k_{0}+3}}\leq \frac{1}{2^{k_{0}}},   
\]
a contradiction, and we are done.
\end{proof}
We note that all \emph{cliquish} functions on the reals are the sum of two quasi-continuous functions \cites{quasibor2, malin}, 
but this property cannot be proved in the base theory (and much stronger systems), as discussed in \cite{dagsamXIV}*{\S2.8.3}.

\smallskip

In conclusion, graph continuity and related notions can be rather \emph{wild} in light of Theorem \ref{lonk} and \ref{lonk2}, but also rather \emph{tame} in light of Theorem \ref{lonk3}.

\subsection{Provable from Kleene's quantifier $(\exists^{2})$}\label{klef2}
In this section, we show that the existence of the supremum functional (Def.\ \ref{KZ}) for weak continuity notions as in Def.\ \ref{WC} can be established using $(\exists^{2})$ and weaker axioms.  
Our main point is that the `tame' notions from Def.\ \ref{WC} occur side-by-side in the literature with the `wild' notions from Def.\ \ref{KY}, which give rise to supremum functionals that yield $(\exists^{3})$. 
We recall that $(\exists^{3})$ (resp.\ $(\exists^{2})$) is the higher-order version of second-order arithmetic $\Z_{2}$ (resp.\ arithmetical comprehension $\ACA_{0}$) from Friedman-Simpson RM.  

\smallskip

First of all, we investigate the following weak continuity notions which mostly stem from topology. 
In particular, continuous functions are characterised topologically by the inverse image of open sets being open.  Weakening the latter requirement 
yields weak continuity notions.  We recall Remark \ref{hio} on virtual existence and note that the inverse image $f^{-1}(E):=\{x \in \R : f(x)\in E\}$ makes sense in $\RCAo$.  
\bdefi[Weak continuity II]\label{WC}
Let $f:[0,1]\di \R$ and $S\subset \R$ be given, 
\begin{itemize}
\item $S$ is \emph{semi-open} if $S\subset \overline{\INT({S})}$ and semi-closed if $\INT(\overline{S})\subset S$.
\item $S$ is an \emph{$\alpha$-set} in case we have the inclusion $S\subset \INT\left(\overline{\INT({S})}\right)$,
\item $S$ is \emph{regular closed} in case we have $S= \overline{\INT({S})}$,
\item $S$ is an \emph{$\mathcal{A}$-set} if $S=O\cap C$ for open $O$ and regular closed $C$,
\item $S$ is an \emph{$\mathcal{AB}$-set} if $S=O\cap C$ for open $O$ and semi-closed semi-open $C$,
\item $f$ is \emph{s-continuous}\footnote{We use `s-continuity' as the full name was used by Baire for a different notion.} if for any open $G\subset \R$, $f^{-1}(G)$ is semi-open,  
\item $f$ is \emph{$\alpha$-continuous} if for any open $G\subset \R$, $f^{-1}(G)$ is an $\alpha$-set,  
\item $f$ is \emph{$\mathcal{A}$-continuous} if for any open $G\subset \R$, $f^{-1}(G)$ is an $\mathcal{A}$-set,  
\item $f$ is \emph{$\mathcal{AB}$-continuous} if for any open $G\subset \R$, $f^{-1}(G)$ is an $\mathcal{AB}$-set,  
\item $f$ is \emph{contra-continuous} if for open $G\subset \R$, $f^{-1}(G)$ is closed,
\item $f$ is SR-continuity if for open $G\subset \R$, $f^{-1}(G)$ is semi-open and semi-closed,
\item $f$ is \emph{quasi-continuous} if for all $x\in [0,1]$, $ \epsilon > 0$, and open neighbourhood $U$ of $x_{0}$, 
there is ${ (a, b)\subset U}$ with $(\forall y\in (a, b)) (|f(x)-f(y)|<\eps)$.
\end{itemize}
\edefi
\noindent
The tame weak continuity notions from Def.\ \ref{WC} also stem from decompositions of continuity.  As an example, continuity is equivalent to the combination of 
$\alpha$-continuity and $\mathcal{A}$-continuity (\cite{tong3}).  The remaining notions are in e.g.\ \cite{tong3, DONT3, DONT4, dont, DONT2}.

\smallskip

Secondly, the weak continuity notions from Def.\ \ref{WC} are relatively tame.  
\begin{thm}[$\ACAo$]\label{TWC}
The supremum functional exists for any $f:[0,1]\di \R$ which is either $\alpha$-, $s$-, $\mathcal{A}$-, $\mathcal{AB}$-, quasi-, contra- or SR-continuous.  
\end{thm}
\begin{proof}
First of all, for $S\subset [0,1]$ an $\alpha$-set, we have the following, by definition:
\[\textstyle
S\ne \emptyset \di \INT(S)\ne \emptyset \di (\exists x\in [0,1], N\in \N)( B(x, \frac{1}{2^{N}})\subset S)\di (\exists r\in \Q\cap [0,1])(r\in S).  
\]
Hence, $\exists^{2}$ suffices to decide if an $\alpha$-set $S$ is empty or not.  Similarly, fix $q\in \R$ and $\alpha$-continuous $f:[0,1]\di \R$ and suppose $(\exists x \in [0,1])(f(x)>q)$.
Now fix $k_{0}\in \N$ and $x_{0}\in [0,1]$ such that $f(x_{0})-\frac{1}{2^{k_{0}}}>q$.  Moreover for $G_{0}=B(x_{0}, \frac{1}{2^{k_{0}}})$, $f^{-1}(G_{0})$ is semi-open and non-empty (as $x_{0}\in G_{0}$).  
Hence, there is $r_{0}\in \Q\cap [0,1]$ in $f^{-1}(G_{0})$ by the previous.  By definition, this implies $f(r_{0})>q$, i.e.\ we have
\be\label{flank}
(\exists x\in [0,1])(f(x)>q)\asa (\exists r\in [0, 1]\cap \Q)(f(r)>q),
\ee
and $\exists^{2}$ suffices to decide the right-hand side of \eqref{flank}. 
In this way, one readily finds the supremum of $f$, which could be the special value $+\infty$, using the usual interval-halving technique. 
The very same proof goes through for most of the other topological continuity notions.  In particular, if $S\ne\emptyset$ is regular closed, then $\INT(S)\ne\emptyset$.    
Moreover, \eqref{flank} holds for quasi-continuity `by definition'. 

\smallskip

Secondly, for contra-continuity, we directly establish \eqref{flank}.
Hence, assume $(\exists x\in [0,1])(f(x)>q)$ and consider the open set $[ -\infty, q )$.  By assumption $f^{-1}([-\infty, q))=\{ x\in [0,1] : f(x)<q \}$ is closed and hence the complement $\{  x\in [0,1] :f(x)\geq q \}$ is open.  
Since the latter is non-empty by assumption, $(\exists r\in \Q\cap [0,1])(f(r)\geq q )$ follows by the definition of open set.  
\end{proof}
Finally, we show that $(\exists^{2})$ is not really needed for the results in this section. 
By \cite{kohlenbach2}*{Theorem 3.15}, $\RCAo+\FF$ is at the level of the Big Five system $\WKL_{0}$. 
Details on Tait's \emph{fan functional} as in $\FF$, may be found in \cite{nota}.
\begin{cor}\label{taitdied}
In Theorem \ref{TWC}, the axiom $(\exists^{2})$ can be replaced by: 
\be\tag{\textsf{FF}}
(\exists \Omega^{3})(\forall Y^{2}\in C(2^{\N}))( \forall h, g\in 2^{\N})(\overline{h}\Omega(Y)=\overline{g}\Omega(Y)\di Y(h)=Y(g) ),
\ee
where `$Y^{2}\in C(2^{\N})$' means that $Y^{2}$ is continuous on $2^{\N}$ in the usual sense. 
\end{cor}
\begin{proof}
First of all, $\FF$ (trivially) implies that all continuous functions on Cantor space are uniformly continuous, which yields weak K\"onig's lemma by \cite{dagsamXIV}*{\S2}.
Based on the latter, the supremum of continuous functions therefore \emph{exists}, and what remains to show is that the associated supremum functional can (somehow) be defined using $\FF$. 

\smallskip

Secondly, to establish the corollary, we use the law of excluded middle as in $(\exists^{2})\vee \neg(\exists^{2})$.  In case $(\exists^{2})$ holds, use the proof of Theorem \ref{TWC}.
In case $\neg(\exists^{2})$, all $\R\di \R$- and $\N^{\N}\di \N$-functions are continuous by \cite{kohlenbach2}*{Prop.\ 3.7 and 3.12}.
Hence, it suffices to define the supremum functional for continuous functions on $[0,1]$. 
Fix continuous $f:[0,1]\di \R$ and consider $Y:(\N^{\N}\times \N)\di \N$ defined as follows $Y(g, k):=[f(\r(\tilde{g}))](k)$, where $\tilde{g}(n)$ is $1$ if $f(n)>0$ and $0$ otherwise, $\r(g):= \sum_{n=0}^{+\infty}\frac{{g}(n)}{2^{n+1}}$, and $[x](k)$ is the $k$-th approximation of the real $x$, coded as a natural number.  By assumption, $\lambda g.Y(g, k)$ is continuous on $2^{\N}$ and use $\FF$ to define $N_{k}:= \Omega(\lambda g.Y(g, k))$ and $M_{k}:= \max_{i\leq {2^{N_{k}}}}f(\r(\sigma_{i}*00) )$ where $(\sigma_{i})_{i\leq k}$ lists all ($2^{k}$ many) finite binary sequence of length $k$. 
Clearly, $M_{k}$ is the $k$-th approximation of $\sup_{x\in [0,1]}f(x)$ and we are done. 
\end{proof}
In conclusion, the supremum functional for a number of weak continuity notions, namely as in Def.\ \ref{WC}, is rather \emph{tame} by Theorem \ref{TWC}.  
By contrast, the supremum functional for weak continuity notions as in Def.\ \ref{KY} is rather \emph{wild} by Theorem \ref{WC}, despite all these notions existing side by side in the literature.  
This establishes the Janus-face nature of weak continuity notions. 

\subsection{Feferman's projection principle and beyond}\label{lef}
We show that the \emph{third-order} supremum principle for certain wild weakly continuous functions implies or is equivalent to strong axioms, including Feferman's projection principle and even full second-order arithmetic.  We recall our different formulations of the supremum \textsf{(sup1)}-\textsf{(sup3)} from Section \ref{klintro}.

\smallskip

In more detail, the equivalences in Section \ref{klef} involve fourth-order objects, like the functional $E$ in $(\exists^{3})$ and the supremum functional $\Xi$.  
For various reasons, one might prefer third-order theorems, and the equivalences in Section \ref{klef} for $(\exists^{3})$ turn out to have third-order counterparts as follows.
\begin{itemize}
\item The supremum principle using \textsf{(sup1)} for certain wild weakly continuous functions on $[0,1]$ is equivalent to $\BOOT$ (Section \ref{lef1}).
\item The supremum principle using \textsf{(sup3)} for certain wild weakly continuous functions on $[0,1]\times [0,1]$ implies $\SIXK$ (Section \ref{beyo}).
\end{itemize}
 We recall that $\BOOT$ is our version of Feferman's projection axiom \textsf{Proj1} from \cite{littlefef}. 
 The former (and the latter) is `highly explosive' following Theorem \ref{timo}, i.e.\ when combined with certain comprehension axioms, one obtains much stronger comprehension axioms.
 
%

\subsubsection{Equivalences for Feferman's projection principle}\label{lef1}
We show that Darboux's supremum principle based on \textsf{(sup1)} for certain weakly continuous functions on $[0,1]$ is equivalent to $\BOOT$; we now introduce the latter.

\smallskip

First of all, the late Sol Feferman was a leading Stanford logician with a life-long interest in foundational matters (\cite{fefermanlight, fefermanmain}), esp.\ so-called predicativist mathematics following Russell and Weyl (\cite{weyldas}). 
As part of this foundational research, Feferman introduced the `projection principle' \textsf{Proj}$_{1}$ in \cite{littlefef}, a third-order version of $(\exists^{3})$ that is impredicative and highly explosive\footnote{Feferman's predicative system $\textup{VFT}+(\mu^{2})$ from \cite{littlefef} is similar to $\ACAo$.  The former plus  \textsf{Proj1} implies $\Z_{2}$, as noted in \cite{littlefef}*{\S5}.}.
In this section, we obtain equivalences between our version of $\textsf{Proj1}$ and the supremum principle (Princ.~\ref{SIP}) for weak continuity classes from Def.\ \ref{KY}.  
In particular, we study the following principle, which is Feferman's principle \textsf{Proj}$_{1}$ formulated in the language of higher-order RM.
\begin{princ}[$\BOOT$] For $Y^{2}$, there is $X\subset \N$ such that 
\be\label{EZ}
(\forall n\in \N)\big[n\in X\asa (\exists f\in \N^{\N})(Y(f, n)=0)  \big].
\ee
\end{princ}
The name of this principle derives from the verb `to bootstrap'.  Indeed, $\BOOT$ is third-order and weaker than $(\exists^{3})$, but Theorem \ref{timo} implies that this principle is \emph{explosive}, yielding as it does $\Pi_{k+1}^{1}$-comprehension when combined with a comprehension functional for $\Pi_{k}^{1}$-formula.  In other words, $\BOOT$ bootstraps itself along the second-order comprehension hierarchy so central to mathematical logic (see \cite{sigohi}).  We have obtained equivalences between $\BOOT$ and convergence theorems for nets in \cites{samph, samhabil}, but the supremum principle (Princ.\ \ref{SIP}) is much more natural.         

\smallskip

Secondly, by Theorem \ref{frank}, $\BOOT$ is equivalent to the usual supremum principle for the weak continuity classes from Def.\ \ref{KY}.
\begin{princ}[Supremum principle for $\Gamma$]\label{SIP}
For any $f:\R\di \R$ in the function class $\Gamma$ and $p, q\in \Q$, $\sup_{x\in [p, q]}f(x)$ exists\footnote{To be absolutely clear, the supremum principle states the existence of a sequence $(x_{n,m})_{n,m\in \N}$ in $\R\cup\{+\infty, -\infty\}$ such that $x_{n,m}=\sup_{x\in [q_{n}, q_{m}]}f(x)$ where $(q_{n})_{n\in \N} $ enumerates $\Q$.}.  
\end{princ}
We stress that while the first equivalence is perhaps not as surprising, the equivalences for supremum principles involving weak continuity notions are important, in our opinion. 
\begin{thm}[$\ACAo$]\label{frank} The following are equivalent. 
\begin{itemize}
\item $\BOOT$, 
\item the supremum principle for all $\R\di \R$-functions.
\item the supremum principle for graph continuous functions on $\R$.
\item the supremum principle for graph continuous $[0,1]\di [0,1]$-functions.
\item the supremum principle for the weak continuity notions in Theorem~\ref{lonk}. 
\end{itemize}
\end{thm}
\begin{proof}
First of all, we repeat (the well-known fact) that $(\exists^{2})$ can convert a real $x\in [0,1]$ to a unique\footnote{In case of two binary representations, chose the one with a tail of zeros.} binary representation.  
Similarly, an element of Baire space $f\in \N^{\N}$ can be identified with its graph as a subset of $\N\times \N$, which can be represented as an element of Cantor space.  
Over $\ACAo$, $\BOOT$ is thus equivalent to the statement that for any $f:(\R\di \N)\di \N$, there is $X\subset \N$ such that for all $n\in \N$:
\be\label{alti2}
(\exists x\in [0,1])(f(x,n)=0)\asa n\in X.
\ee
Secondly, to show that $\ACAo+\BOOT$ proves the supremum principle for any function $f:\R\di \R$, $\BOOT$ in the form \eqref{alti2} yields $X\subset \Q^{3}$ with
\[ 
(\forall p,q, r\in \Q)\big[  (\exists x\in [p,  q])(f(x)>r)\asa (p, q, r)\in X ].
\]
Now use $\exists^{2}$ and the usual interval-halving method to find the relevant supremum. 

\smallskip

Secondly, let $f:(\R\times \N)\di \N$ be given and note that in light of \eqref{alti2}, we may assume $(\forall q\in \Q, n\in \N)(f(q, n)\ne 0)$.
Indeed, one can use $\exists^{2}$ to decide the latter formula and redefine $f$ if necessary.  Now define $g_{f}:\R\di \R$ as follows:
\[
g_{f}(x):=
\begin{cases}
1 & x\in (n+1,n+2 )\wedge (\exists q\in \Q)( x+q\in [0,1]\wedge f(x+q, n)=0   )\\
0   &  \textup{otherwise}
\end{cases}.
\]
Note that $g_{f}(q)=0$ for $q\in\Q$, i.e.\ $\R\times \{0\}\subset \overline{G(g_{Y})} $.
Hence, the `zero-everywhere' function $g_{0}$ is such that $G(g_{0})\subset \overline{G(g_{f})}$, i.e.\ $g_{f}$ is graph continuous. 
By definition, we have for all $n \in \N $:
\be\label{heng}\textstyle
(\exists x\in [0,1])(f(x, n)=0)\asa [0< \sup_{x\in [n,n+1]}g_{f}(x)], 
\ee
i.e.\ $\BOOT$ follows from the supremum principle for graph continuous functions.
As in the proof of Theorem \ref{lonk}, one verifies that $g_{f}$ also satisfies all the other weak continuity notions listed in Theorem \ref{lonk}. 
\end{proof}
One can obtain similar results for the existence of an \emph{oscillation function}, going back to Riemann and Hankel (\cites{riehabi,hankelwoot}), but the latter
is defined in terms of the supremum operator anyway.    More interestingly, one could generalise the supremum principle from intervals to arbitrary sets of reals. 

\smallskip

Thirdly, the evaluation of the Riemann integral is a central part of (undergraduate) calculus.  
Basic examples show that this integral can be finite, infinite, or undefined.  The exact case can vary with the end-points, as is clear in light of the following evaluation of the Riemann integral of $\frac1x$:
\[\textstyle
\int_{a}^{b} \frac{dx}{x} =
\begin{cases}
\ln (|b|)-\ln(|a|)  & 0<a< b \vee a<b<0 \\
\quad {\uparrow} & a< 0 <  b  \\
+\infty & 0=a <  b \\
-\infty & a<b=0
\end{cases},
\]
 where the arrow $\uparrow$ is a symbol representing `undefined', standard in computability theory and satisfying\footnote{To be absolutely clear, the formula `$x=y$' is $\Pi_{1}^{0}$ for $x, y\in \overline{\R}$ by Footnote \ref{triffo} (see \cites{simpson2, kohlenbach2} for equality on the reals) and we assume the same complexity for $x, y\in \overline{\R}\cup\{\uparrow\}$.\label{triffo2}} $x\ne {\uparrow}$ for $x\in \overline{\R}$ by fiat.  We now study the following principle where we use the `epsilon-delta' definition of Riemann integration. 
\begin{princ}[Evaluation principle for $\Gamma$]\label{EIP}
For any $f:\R\di \R$ in the function class $\Gamma$ and $p, q\in \Q$, the Riemann integral $\int_{p}^{q}f(t)dt$ can be evaluated\footnote{To be absolutely clear, the evaluation principle states the existence of a sequence $(x_{n,m})_{n,m\in \N}$ in $\R\cup \{+\infty, -\infty, \uparrow\}$ such that $x_{n,m}=\int_{q_{n}}^{q_{m}}f(t)dt$ where $(q_{n})_{n\in \N} $ enumerates $\Q$.}. 
\end{princ}
Apparently, evaluating the Riemann integral is hard, as follows. 
\begin{thm}[$\ACAo$]\label{frank2} The following are equivalent. 
\begin{itemize}
\item $\BOOT$, 
\item the evaluation principle for graph continuous functions on $\R$,
\item the evaluation principle for graph continuous $[0,1]\di [0,1]$-functions,
\item the evaluation principle for the weak continuity notions in Theorem~\ref{WC}. 
\end{itemize}
\end{thm}
\begin{proof}
To obtain $\BOOT$ from the other principles, Dirichlet's function $\mathbb{1_{Q}}$ is readily shown to satisfy $\int_{0}^{1}\mathbb{1_{Q}}(t)dt={\uparrow}$ in $\ACAo$.
In exactly the same way, in case $f:\R\di \R$ is not identically $0$ on $[0,1]$, $g_{f}$ from \eqref{kruk} is not Riemann integrable.  
Thus, instead of \eqref{heng}, we have for all $n\in \N$ that:
\be\label{hengk}\textstyle
(\forall x\in [0,1])(f(x, n)\ne 0)\asa \int_{n}^{n+1}g_{f}(t)dt =0, 
\ee
which is as required for $\BOOT$, similar to the proof of Theorem \ref{frank}.

\smallskip

To derive the other principles from $\BOOT$, consider the usual definition of Riemann integrability on $[0,1]$:  
\be\label{fing}\textstyle
 (\forall k\in \N)(\exists N\in  \N)\underline{(\forall P, Q)( \|P\|, \|Q\|< \frac{1}{2^{N}} \di |S(f, P)-S(f, Q)|\leq \frac{1}{2^{k}})},
\ee
where $P=(0, x_{0},t_{0}, x_{1}, t_{1}, \dots , x_{k}, t_{k}, 1)$ is a partition of $[0,1]$, the real $S(f, P):=\sum_{i}f(t_{i})(x_{i+1}-x_{i})$ is the associated Riemann sum, and $\|P\|=\max_{i}(x_{i+1}-x_{i})$ is the mesh.  
Let $A(k, N)$ be the underlined formula in \eqref{fing} and note that \eqref{alti2} applies, modulo the removal of arithmetical quantifiers using $(\exists^{2})$.  
By $\BOOT$, there is $X\subset \N^{2}$ such that $(k,N)\in X\asa A(k,N)$.  In this way, Riemann integrability for $f$ is equivalent to $(\forall k\in \N)(\exists N\in \N)[(k, N)\in X]$, which is decidable given $(\exists^{2})$. 
Replacing $[0,1]$ by $[p, q]$ for $p, q\in \Q$, one obtains the evaluation principle.  
\end{proof}
Finally, we show that $\BOOT$ is explosive when combined with comprehension functionals as in Theorem \ref{timo}.  
In the latter, $\SIXk$ is $\RCA_{0}$ plus comprehension for $\Pi_{k}^{1}$-formulas, i.e.\ for $\varphi(n)$ as in the latter, the set $\{n\in \N: \varphi(n)\}$ exists.  
On a technical note, the \emph{Kleene normal form theorem} is provable in $\ACA_{0}$ (\cite{simpson2}*{V.5.4}) and implies that $\Pi_{k}^{1}$-formulas can be represented in the so-called \emph{Kleene normal form}.  
Using this normal form, $(\SS_{k}^{2})$ expresses the existence of a functional $\SS_{k}^{2}$ which decides the truth of $\Pi_{k}^{1}$-formulas given in their Kleene normal form. 
We define $\SIXK$ as $\RCAo+(\SS_{k}^{2})$ and have the following theorem.  A more detailed discussion of $(\SS_{k}^{2})$ is in Remark \ref{kol} at the end of this section. 
\begin{thm}\label{timo}~
\begin{itemize}
\item The system $\ACAo+\BOOT$ proves $\FIVE$.
\item The system $\SIXK+\BOOT$ proves $\SIXko$ for $k\geq 1$. 
\end{itemize}
\end{thm}
\begin{proof}
For the first item, let $\varphi(n)$ be a $\Pi_{1}^{1}$-formula.  By the Kleene normal form theorem (\cite{simpson2}*{V.5.4}), there is $f\in \N^{\N}$ such that 
\[
(\forall n\in \N)\big[\varphi(n)\asa (\exists g\in \N^{\N})\underline{(\forall m\in \N)(f(\overline{g}m, n)=0)}\big]. 
\]
The underlined formula is decidable using $\exists^{2}$, i.e.\ there is $Y^{2}$ such that 
\[
(\forall n\in \N)\big[\varphi(n)\asa (\exists g\in \N^{\N})(Y(g, n)=0)\big], 
\]
and $\BOOT$ yields $X\subset \N$ such that $\varphi(n)\asa n\in X$ for all $n\in \N$. 
The proof of the second item proceeds in the same way \emph{mutatis mutandis}. 
\end{proof}
While the `distance' between $\SIXK$ and $\SIXko$ is not as immense as that between $(\exists^{2})$ and $(\exists^{3})$, the principle $\BOOT$ still enables a considerable leap in logical strength, when provided with the axiom $(\SS_{k}^{2})$. 

\smallskip

Finally, we discuss the axiom $(\SS_{k}^{2})$ in more detail as follows. 
\begin{rem}\label{kol}\rm
The functional $\SS^{2}$ in $(\SS^{2})$ is called \emph{the Suslin functional} (\cite{kohlenbach2}):
\be\tag{$\SS^{2}$}
(\exists\SS^{2}:\N^{\N}\di \N)(\forall f \in \N^{\N})\big[  (\exists g\in \N^{\N})(\forall n\in \N)(f(\overline{g}n)=0)\asa \SS(f)=0  \big].
\ee
The system $\FIVE^{\omega}\equiv \RCAo+(\SS^{2})$ proves the same $\Pi_{3}^{1}$-sentences as $\FIVE$ by \cite{yamayamaharehare}*{Theorem 2.2}.   
By definition, the Suslin functional $\SS^{2}$ can decide whether a $\Sigma_{1}^{1}$-formula as in the left-hand side of $(\SS^{2})$ is true or false.   We similarly define the functional $\SS_{k}^{2}$ which decides the truth or falsity of $\Sigma_{k}^{1}$-formulas; we also define 
the system $\SIXK$ as $\RCAo+(\SS_{k}^{2})$, where  $(\SS_{k}^{2})$ expresses that $\SS_{k}^{2}$ exists.  
We note that the Feferman-Sieg operators $\nu_{n}$ from \cite{boekskeopendoen}*{p.\ 129} are essentially $\SS_{n}^{2}$ strengthened to return a witness (if existant) to the $\Sigma_{n}^{1}$-formula at hand.  
\end{rem}

\subsubsection{Beyond the projection principle}\label{beyo}
By Theorem \ref{frank}, $\BOOT$ is equivalent to Darboux's supremum principle on the unit interval for the weak continuity classes from Def.\ \ref{KY}.
We now switch to the plane and assume the supremum is given as a function, namely as in $\textsf{(sup2)}$ from Section \ref{klintro}.
\begin{princ}[Supremum principle for $\Gamma^{2}$]\label{SSP}
For any $f:\R^{2}\di \R$ in the function class $\Gamma^{2}$ and $p, q\in \Q$, the function $\lambda x.\sup_{y\in [p, q]}f(x, y)$ exists\footnote{To be absolutely clear, Princ.\ \ref{SSP} states the existence of $\Phi:(\R\times \Q^{2})\di \R\cup\{+\infty, -\infty\}$ such that $\Phi(x, p, q)=\sup_{y\in [p,q]}f(x, y)$ for any $x\in \R$ and $p, q\in \Q$.}.  
\end{princ}
We shall also need the following `two dimensional' version of $\BOOT$.
\begin{princ}[$\BOOT_{2}$] For $Y^{2}$, there is $X\subset \N$ such that 
\be\label{EZER}
(\forall n\in \N)\big[n\in X\asa (\exists f\in \N^{\N})(\forall g\in \N^{\N})(Y(f,g, n)=0)  \big].
\ee
\end{princ}
Clearly, $\RCAo+\BOOT_{2}$ proves $\FIVE$ and $\ACAo+\BOOT_{2}$ proves $\SIX$. 
We now have the following theorem where `supremum principle' refers to Principle~\ref{SSP}. 
\begin{thm}[$\ACAo$]\label{franker} The higher items imply the lower ones.  
\begin{itemize}
\item the supremum principle for all $\R^{2}\di \R$-functions.
\item the supremum principle for graph continuous $\R^{2}\di \R$-functions.
\item the supremum principle for graph continuous $[0,1]^{2}\di [0,1]$-functions.
\item The principle $\BOOT_{2}$.
\end{itemize}
We may replace `graph continuous' by any of the items in Def.~\ref{WC}.
\end{thm}
\begin{proof}
By Theorem \ref{frank}, we may freely use $\BOOT$.
As in the proof of Theorem~\ref{frank}, $\BOOT_{2}$ is equivalent to the statement that for any $f:(\R^{2}\times\N)\di \N$, there is $X\subset \N$ such that for all $n\in \N$:
\be\label{alti22}
(\exists x\in [0,1])(\forall y\in [0,1])(f(x,y,n)\ne 0)\asa n\in X.
\ee
Now let $f:(\R^{2}\times \N)\di \N$ be given and
 define $g_{f}:\R^{2}\di \R$ as follows:
\[
g_{f}(x, y):=
\begin{cases}
1 & y\in (n+1,n+2 )\setminus \Q\wedge (\exists q\in \Q)( y+q\in [0,1]\wedge f(x, y+q,n)=0   )\\
0   &  \textup{otherwise}
\end{cases}.
\]
Note that $g_{f}(q, r)=0$ for $q, r\in\Q$, i.e.\ $[(\R\times \R)\times \{0\}]\subset \overline{G(g_{f})} $.
Hence, the `zero-everywhere' function $g_{0}$ on $[0,1]^{2}$ is such that $G(g_{0})\subset \overline{G(g_{f})}$, i.e.\ $g_{f}$ is graph continuous. 
By definition, we have for all $n \in \N $:
\begin{align}\label{heng2}\textstyle
(\exists x\in [0,1])(\forall y\in [0,1])&(f(x,y, n)\ne0)\notag\\ 
&\asa \\
\textstyle(\exists x\in [0,1])\big[[0= \sup_{y\in [n+1,n+2]}g_{f}(x, y)] &\wedge (\forall z\in [0,1]\cap \Q)(f(x, z, n)\ne0) \big]. \notag
\end{align}
The bottom formula in \eqref{heng2} is amenable to $\BOOT$ as the formula in large square brackets is arithmetical.  
Thus, $\BOOT_{2}$ follows from the supremum principle for graph continuous functions as in Principle \ref{SSP}.
As in the proof of Theorem \ref{lonk}, one verifies that $g_{f}$ also satisfies all the other weak continuity notions listed in Definition~\ref{WC}. 
\end{proof}
Let $\BOOT_{k}$ for $k\geq 3$ be the obvious generalisation of $\BOOT_{2}$ to $k-1$ quantifier alternations.  
We note that $\BOOT_{k}\di \SIXk$ over $\RCAo$ while $\ACAo+\BOOT_{k}$ proves $\SIXko$.
\begin{cor}[$\ACAo$]\label{felf}
The higher items imply the lower ones.  
\begin{itemize}
\item The supremum principle for graph continuous $[0,1]^{2}\di [0,1]$-functions.
\item $\BOOT_{k}$ for $k\geq 3$.
\end{itemize}
We may replace `graph continuous' by any of the items in Def.~\ref{WC}.
\end{cor}
\begin{proof}
By Theorem \ref{franker}, we may freely use $\BOOT_{2}$.
We shall derive $\BOOT_{3}$ as in \eqref{alti223} after which the general case is straightforward.
Indeed, as in the proof of Theorem~\ref{frank}, $\BOOT_{3}$ is equivalent to the statement that for any $h:(\R^{3}\times\N)\di \N$, there is $X\subset \N$ such that for all $n\in \N$:
\be\label{alti223}
(\exists x\in [0,1])(\forall y\in [0,1])(\exists z\in [0,1])(h(x,y, z,n)\ne 0)\asa n\in X.
\ee
Now let $h:(\R^{3}\times \N)\di \N$ be given and define $g_{h}:\R^{2}\di \R$ as follows:
\[
g_{h}(x, y):=
\begin{cases}
1 & y\in (n+1,n+2 )\setminus \Q\wedge (\exists q\in \Q)( y+q\in [0,1]\wedge h(O(x),E(x), y+q,n)=0   )\\
0   &  \textup{otherwise}
\end{cases},
\]
where $O: \R\di 2^{\N}$ and $E: \R\di 2^{\N}$ are such that if $\alpha$ is the binary\footnote{Choose the expansion with a tail of zeros in case there are multiple.\label{labbe}} expansion of $x$, then $\alpha(2n)=O(x)(n)$ and $\alpha(2n+1)=E(x)(n)$ for any $n\in \N$.  As in the proof of Theorem \ref{franker}, $g_{h}$ is graph continuous, i.e.\ we have access to $\lambda x.\sup_{y\in [p, q]}g_{h}(x, y)$.  Now consider the following equivalence
\begin{align}\label{heng3}\textstyle
(\exists x\in [0,1])(\forall y\in [0,1])&(\exists z\in [0,1])(h(x,y,z, n)\ne0)\notag\\ 
&\asa \\
\textstyle(\exists x\in [0,1])(\forall y\in [0,1])&\left[\begin{array}{c}1= \sup_{z\in [n+1,n+2]}g_{h}(J(x, y), z)] \\ \vee  \\(\exists z\in [0,1]\cap \Q)(h(x, y,z, n)\ne0) \big]\end{array}\right] \notag
\end{align}
where $J(x, y)$ is defined as $\sum_{n=0}^{\infty} \frac{\alpha(2n)}{2^{2n}}+\frac{\beta(2n+1)}{2^{2n+1}} $ for $\alpha, \beta\in 2^{\N}$ the binary$^{\ref{labbe}}$ representations of $x, y$.  
The bottom formula in \eqref{heng3} is amenable to $\BOOT_{2}$ as the formula in large square brackets is arithmetical.  
Thus, $\BOOT_{3}$ follows from the supremum principle for graph continuous functions as in Principle \ref{SSP}.  
As in the proof of Theorem \ref{lonk}, one verifies that $g_{f}$ also satisfies all the other weak continuity notions listed in Theorem \ref{lonk}. 
\end{proof}
%
%
We finish this section with some conceptual remarks generalising our results.
\begin{rem}[Darboux versus Weierstrass]\rm
Continuous functions on the unit interval have the intermediate value property and attain their maximum value.  
The former is also called the \emph{Darboux property} following \cite{darb} and the associated function class has been studied in some detail (\cites{malin, gibs, ellis, poki, faster}).
One could similarly say that a bounded $[0,1]\di [0,1]$-function has the `Weierstrass property' in case it attains its maximum value, i.e.\ $(\exists x\in [0,1])(\forall y\in [0,1])(f(y) \leq f(x))$.
Principles akin to \textsf{(sup1)-(sup3)} for the Weierstrass property of weakly continuous functions seem to yield $\BOOT$, $(\exists^{3})$, and $\Z_{2}$ in the same way as above.  
\end{rem}
\begin{rem}[Riemann integration]\rm
Fubini's theorem (\cite{taomes}) provides rather weak conditions under which a double integral $\iint_{X\times Y}f(x,y)d(x,y)$ reduces to the composition of single integrals.  
This theorem is true for the Lebesgue integral and (only) holds for the Riemann integral if we additionally assume that the single integrals $\int_{0}^{1} f(x, y)dy$ and $\int_{0}^{1} f(x, y)dx$ exists for all $x, y\in [0,1]$.  Special cases of the double (Riemann) integral have been studied by Euler, Abel, and Riemann.  
Hence, the following principle is rather natural.
\begin{princ}[Riemann integration principle for $\Gamma^{2}$]\label{RIP}
For any $f:\R^{2}\di [0,1]$ in the function class $\Gamma^{2}$, the function $\lambda x. \int_{0}^{1} f(x, y)dy$ exists\footnote{To be absolutely clear, Princ.\ \ref{RIP} states the existence of $\Phi:\R\di (\R\cup \{\uparrow\})$ such that $\Phi(x)=\int_{0}^{1} f(x, y)dy$ if the latter exists, and $\Phi(x)=\uparrow$ otherwise.}.  
\end{princ}
One obtains $\BOOT_{k}$ from this principle in the same way as for Corollary \ref{felf}.
\end{rem}
\begin{ack}\rm 
We thank Anil Nerode for his valuable advice, especially the suggestion of studying nsc for the Riemann integral, and discussions related to Baire classes.
We thank Dave L.\ Renfro for his efforts in providing a most encyclopedic summary of analysis, to be found online.  
Our research was supported by the \emph{Klaus Tschira Boost Fund} via the grant Projekt KT43.
%
We express our gratitude towards the latter institution.    
\end{ack}

\bye